\documentclass[11pt]{amsart}

\usepackage[T1]{fontenc}
\usepackage{lmodern}
\usepackage{amsmath,amssymb,amsthm,mathtools}
\usepackage[a4paper,margin=32mm]{geometry}
\usepackage[hidelinks]{hyperref}

\newtheorem{proposition}{Proposition}[section]
\newtheorem{theorem}[proposition]{Theorem}
\newtheorem{lemma}[proposition]{Lemma}

\newcommand{\Sym}{\mathrm{Sym}}
\newcommand{\scalar}[2]{\left\langle #1,#2\right\rangle_{q,t}}
\newcommand{\Fcal}{\mathcal F}
\newcommand{\Dop}{\mathsf D_t}

\newcommand{\Pcal}{\mathcal P}

\title[Shifted Macdonald Polynomials and a Deformed Product]
{Shifted Macdonald Polynomials and the\\
$(q,t)$-Deformed Goulden--Jackson Product}
\author{Jean-Yves Thibon}

\address{Laboratoire d’Informatique Gaspard-Monge,
Université Gustave Eiffel,
CNRS, ESIEE Paris,
F-77454 Marne-la-Vallée, France}
\email[Jean-Yves Thibon]{jean-yves.thibon@univ-eiffel.fr} 


\date{}

\begin{document}
\begin{abstract}
In a preceding article, we introduced stable symmetric series encoding
simultaneously the normalized conjugacy classes of all symmetric
groups.  The same rational series remain stable for the Jack-deformed
Goulden--Jackson product.  We investigate their two-parameter
Macdonald analogue.  Starting from the $(q,t)$-deformed class product 
(dual to the coproduct diagonal on the $J$-basis),
we construct the unique infinite series
whose multiplication realizes any shifted Macdonald eigenvalue.  In
contrast with the classical and Jack cases, this transform is no
longer multiplication by a fixed explicit series.  We identify it
with a composition of a Cauchy multiplication, the integral nabla
operator, and a simple diagonal operator.  We then determine the
series realizing the Nazarov--Sklyanin operators $A^{(k)}$, derive a
single generating series for all column partitions, and compare our construction
 with the Macdonald characters and Theta operators of Ben Dali and
D'Adderio.
\end{abstract}
\maketitle

\section{Introduction}

This paper is a direct continuation of
\cite{ThibonStable}.  For a homogeneous symmetric function $f$ of
degree $m$, that article introduced the completed symmetric series
\begin{equation}\label{eq:intro-classical-hat}
 \widehat f=\frac{m!f}{(1-p_1)^{m+1}}.
\end{equation}
When $f=p_\mu$, its degree-$n$ component is
\begin{equation}\label{eq:intro-class-component}
 \widehat p_\mu^{[n]}
 =(n)_{m}p_\mu p_1^{n-m}, \quad (n)_m :=n(n-1)\cdots (n-m+1),
 \qquad n\geq m=|\mu|,
\end{equation}
and is the cycle index of a normalized stable conjugacy
class.  Multiplication of these series for the class product 
recovers the Ivanov--Kerov algebra.  For
example,
\begin{equation}\label{eq:intro-classical-square}
 \widehat p_2\times\widehat p_2
 =\widehat p_{22}+4\widehat p_3+2\widehat p_{11}.
\end{equation}

A striking observation of \cite{ThibonStable} is that exactly the
same rational series \eqref{eq:intro-classical-hat}, without changing
their coefficients, are stable for the Jack-deformed
Goulden--Jackson product $\times_\alpha$.  Only their multiplication
law changes; for instance,
\begin{equation}\label{eq:intro-Jack-square}
 \widehat p_2\times_\alpha\widehat p_2
 =\widehat p_{22}+4\widehat p_3
  +2\alpha\widehat p_{11}+2(\alpha-1)\widehat p_2.
\end{equation}
Shifted symmetric functions $f^*$  arise then naturally in both cases  as eigenvalues of $\times$-multiplication by the series $\hat f$: $\widehat s_\mu\times s_\lambda=s_\mu^*(\lambda)s_\lambda$, and
$\widehat p_\mu \times_\alpha P_\lambda^{(\alpha)}={\rm Ch}_\mu^{(\alpha)}(\lambda) P_\lambda^{[\alpha)}$.

This coincidence makes it natural to ask what becomes of the stable
series under the full $(q,t)$-deformation.

The first answer is that the simple rational expression
\eqref{eq:intro-classical-hat} does not survive.  The shifted
Macdonald representative of $f$ is still uniquely determined by its
eigenvalues, but the map $f\mapsto\widehat f$ is no longer
multiplication of $f$ by an explicit universal series.  
%
%Let us make the normalization implicit in this statement explicit.
For $\lambda\vdash n$, we seek a completed series whose homogeneous
component of degree $n$ satisfies
\[
 \widehat P_\mu^{[n]}\times_{q,t}P_\lambda
 =P_\mu^*(q^\lambda;q,t)P_\lambda,
\]
where $P_\mu^*$ has Okounkov's normalization.  
%The factor $n!$ in the definition of $\widehat P_\mu^{[n]}$ is already forced by the
%factor $1/n!$ in the diagonal coproduct, so no factorial occurs in this eigenvalue formula. 

Okounkov's diagonal normalization is useful
for comparison with the interpolation literature, but it is not
adapted to the lowest homogeneous component: if $m=|\mu|$, then
\[
 \widehat P_\mu^{[m]}
 =m!(-1)^m q^{n(\mu')}t^{-2n(\mu)}P_\mu.
\]
We shall therefore also introduce a renormalized transform
$f\mapsto\check f$ for which $\check f^{[m]}=m!f$.  In exponential
normalization it is simply
\[
 \mathbf{\check f}
 =\mathsf D_t\,\mathcal P_{1/(1-q)}\mathsf D_t^{-1}f,
\]
where $\mathcal P_Z$ denotes multiplication by $\sigma_1[ZX]$, and
$\mathsf D_tP_\lambda=t^{-n(\lambda)}P_\lambda$.  
Thus the hat notation retains Okounkov's eigenvalues, whereas the
check notation makes the transform unitriangular with respect to
degree.  After the factorial normalization is removed, we shall prove
that
\begin{equation}\label{eq:intro-Macdonald-transform}
 \mathbf{\widehat f}
 =\mathsf D_t\,\mathcal P_{1/(1-q)}\boldsymbol\nabla f,
\end{equation}
where $\boldsymbol\nabla$ is the integral-form nabla operator.

Formula
\eqref{eq:intro-Macdonald-transform} is already more complicated than
\eqref{eq:intro-classical-hat}: the Cauchy multiplication is followed
by a nonmultiplicative diagonal operator.

Our route to \eqref{eq:intro-Macdonald-transform} begins with the
$\times_{q,t}$-product, dual to the diagonal coproduct
\[
 \Gamma_{q,t}^{(n)}J_\lambda
 =\frac1{n!}J_\lambda\otimes J_\lambda.
\]
We first reconstruct spectrally the series representing Okounkov's
shifted Macdonald polynomials.  We then determine the series representing the
differential  operators of Nazarov--Sklyanin \cite{NazarovSklyanin}.
A Hall--Littlewood Cauchy calculation gives both the
individual series representing $A^{(k)}$ and a generating series for
all $k$.  Okounkov's binomial formula identifies their eigenvalues
with shifted columns at inverted parameters.

Finally, we compare this construction with Ben Dali and D'Adderio's
Macdonald characters \cite{BenDaliDAdderio}.  
Their creation operator $\Gamma(u,v)$ is a
quotient of two Theta series, and their isomorphism from ordinary to
shifted symmetric functions implies exactly
\eqref{eq:intro-Macdonald-transform}.  Thus the two approaches lead
to the same shifted functions: their construction emphasizes creation
operators and positivity, whereas the diagonal coproduct selects the
unique completed series representing these functions as multiplication
eigenvalues.  In particular, their Macdonald character
$\widetilde\theta_\mu^{(q,t)}=p_\mu^*$ is represented here by
$\widehat p_\mu$.

\section{The stable series}
\subsection{The diagonal coproduct}
Let $P_\lambda(X;q,t)$ and $J_\lambda(X;q,t)$ denote respectively the
monic and integral Macdonald functions.  We write
\[
 J_\lambda=c_\lambda(q,t)P_\lambda,
 \qquad
 \scalar{J_\lambda}{J_\lambda}
 =j_\lambda(q,t)=c_\lambda(q,t)c'_\lambda(q,t).
\]
In homogeneous degree $n$, consider the  diagonal coproduct
\begin{equation}\label{eq:diagonal-coproduct}
 \Gamma_{q,t}^{(n)}J_\lambda
 =\frac1{n!}J_\lambda\otimes J_\lambda,
 \qquad \lambda\vdash n.
\end{equation}
Let $\times_{q,t}$ be its dual product for the Macdonald scalar
product.  Orthogonality gives
\begin{equation}\label{eq:diagonal-product}
 J_\lambda\times_{q,t}J_\mu
 =\delta_{\lambda\mu}\frac{j_\lambda(q,t)}{n!}J_\lambda.
\end{equation}

Since $Q_\lambda=J_\lambda/c'_\lambda$, the same identity in the
Macdonald $Q$-basis reads
\begin{equation}\label{eq:diagonal-product-Q}
 Q_\lambda\times_{q,t}Q_\nu
 =\delta_{\lambda\nu}\frac{c_\lambda(q,t)}{n!}Q_\lambda.
\end{equation}

\subsection{The hat transform}

For Schur or Jack functions, knowledge of the explicit form of the stable series was sufficient to reconstruct from scratch the theory of shifted functions. 
In the Macdonald case, low degree computations quickly show
 that this expression is no longer valid, so we have to resort to the existing 
theory of shifted Macdonald functions. 

We use Okounkov's normalization of the interpolation Macdonald
polynomials and abbreviate
\begin{equation}\label{eq:shifted-evaluation-notation}
 P_\mu^*(\lambda)
 :=P_\mu^*(q^\lambda;q,t).
\end{equation}
Thus $P_\mu^*(\lambda)=0$ unless $\mu\subseteq\lambda$.  For every
$n\geq0$, define
\begin{equation}\label{eq:hat-P-component}
  {
 \widehat P_\mu^{[n]}(X)
 :=n!\sum_{\lambda\vdash n}
 \frac{P_\mu^*(\lambda)}{c_\lambda(q,t)}
 Q_\lambda(X;q,t). }
\end{equation}
In particular, $\widehat P_\mu^{[n]}=0$ for $n<|\mu|$.  We regard
\[
 \widehat P_\mu=\sum_{n\geq |\mu|}\widehat P_\mu^{[n]}
\]
as an element of the degree completion of $\Sym$.  Extending
$P_\mu\mapsto\widehat P_\mu$ linearly defines the \emph{hat
transform}
\[
 f\longmapsto\widehat f.
\]
Equivalently, if $f=\sum_\mu f_\mu P_\mu$, its shifted version is
$f^*=\sum_\mu f_\mu P_\mu^*$, and $\widehat f$ is characterized by
the following spectral property.

\begin{proposition}\label{prop:hat-characterization}
For every $\lambda\vdash n$,
\begin{equation}\label{eq:hat-characterization}
  {
 \widehat P_\mu^{[n]}\times_{q,t}Q_\lambda
 =P_\mu^*(\lambda)Q_\lambda. }
\end{equation}
More generally,
\[
 \widehat f^{[n]}\times_{q,t}Q_\lambda
 =f^*(\lambda)Q_\lambda.
\]
The map $f\mapsto\widehat f$ is the unique linear map with this
property.
\end{proposition}

\begin{proof}
Substitute \eqref{eq:hat-P-component} into the left-hand side and use
\eqref{eq:diagonal-product-Q}.  All terms vanish except the one
indexed by $\lambda$, and its two scalar factors cancel:
\[
 n!\frac{P_\mu^*(\lambda)}{c_\lambda}
 \frac{c_\lambda}{n!}Q_\lambda
 =P_\mu^*(\lambda)Q_\lambda.
\]
Since the $Q_\lambda$ form a basis in every degree, the same
calculation also proves uniqueness.
\end{proof}

\subsection{Exponential normalization}
It is convenient to remove the factorials and introduce the
exponentially normalized completed image
\begin{equation}\label{eq:hat-exponential}
 \mathbf{\widehat P}_\mu
 :=\sum_{n\geq|\mu|}\frac{\widehat P_\mu^{[n]}}{n!}.
\end{equation}
Using $Q_\lambda/c_\lambda=P_\lambda/c'_\lambda$, we have
\begin{equation}\label{eq:hat-exponential-spectral}
 \mathbf{\widehat P}_\mu
 =\sum_{\lambda\supseteq\mu}
 \frac{P_\mu^*(\lambda)}{c'_\lambda(q,t)}P_\lambda(X;q,t).
\end{equation}

\subsection{Summation of the series $\mathbf{\widehat P}_\mu$}
Recall Okounkov's $(q,t)$-binomial coefficients
\begin{equation}\label{eq:Okounkov-binomial}
 \left[\begin{matrix}\lambda\\ \mu\end{matrix}\right]_{q,t} 
 :=\frac{P_\mu^*(q^\lambda;q,t)}
          {P_\mu^*(q^\mu;q,t)}.
\end{equation}

The principal-specialization form of Okounkov's binomial theorem is
the following identity.
\begin{lemma}\label{lem:skew-binomial}
For $\mu\subseteq\lambda$,
\begin{equation}\label{eq:skew-binomial}
 t^{n(\mu)-n(\lambda)}
 Q_{\lambda/\mu}\left[\frac1{1-t};q,t\right]
 =\frac{c'_\mu(q,t)}{c'_\lambda(q,t)}
 \left[\begin{matrix}\lambda\\ \mu\end{matrix}\right]_{q,t}.
\end{equation}
Both sides vanish when $\mu\not\subseteq\lambda$.
\end{lemma}

\begin{proof}
%We spell out the translation between the two normalizations involved.
Lascoux and Warnaar introduce the normalized skew Macdonald function
\cite[Eq.~(2.7b)]{LascouxWarnaar}
\begin{equation}\label{eq:normalized-skew-Q-proof}
 \mathcal Q_{\lambda/\mu}(X;q,t)
 :=t^{n(\mu)-n(\lambda)}
   \frac{c'_\lambda(q,t)}{c'_\mu(q,t)}
   Q_{\lambda/\mu}(X;q,t).
\end{equation}
Their notation uses a boldface $Q$; we use $\mathcal Q$ here to avoid
confusion with Macdonald's ordinary dual function $Q_{\lambda/\mu}$.

The generalized Macdonald binomial coefficient is then
\cite[Eq.~(2.15)]{LascouxWarnaar}
\begin{equation}\label{eq:LW-binomial-proof}
 \left[\begin{matrix}\lambda\\ \mu\end{matrix}\right]_{q,t}
 :=\mathcal Q_{\lambda/\mu}
       \left[\frac1{1-t};q,t\right].
\end{equation}
This definition is the skew-function
form of the generalized binomial coefficients of Lassalle and
Okounkov; the equivalence is stated explicitly in
\cite[Eq.~(2.15) and the references preceding it]{LascouxWarnaar}.

Substituting \eqref{eq:normalized-skew-Q-proof} into
\eqref{eq:LW-binomial-proof}, and then using
\eqref{eq:Okounkov-binomial}, gives
\[
 \left[\begin{matrix}\lambda\\ \mu\end{matrix}\right]_{q,t}
 =t^{n(\mu)-n(\lambda)}
   \frac{c'_\lambda}{c'_\mu}
   Q_{\lambda/\mu}\left[\frac1{1-t};q,t\right].
\]
Multiplication by $c'_\mu/c'_\lambda$ is exactly
\eqref{eq:skew-binomial}.

Alternatively, the coincidence of Lassalle's definition \cite{Lassalle} of
the $(q,t)$-binomial coefficients with that of Okounkov is equivalent to the Lemma.

As a normalization check, take $\mu=\varnothing$.  The principal
specialization formula
\cite[Chapter~VI, Eq.~(6.11')]{Macdonald}, followed by the stable
limit, gives
\[
 P_\lambda\left[\frac1{1-t};q,t\right]
 =\frac{t^{n(\lambda)}}{c_\lambda(q,t)}.
\]
Since $Q_\lambda=(c_\lambda/c'_\lambda)P_\lambda$, this yields
\[
 t^{-n(\lambda)}Q_\lambda
       \left[\frac1{1-t};q,t\right]
 =\frac1{c'_\lambda(q,t)},
\]
which is the case $\mu=\varnothing$ of the lemma and independently
checks both the power of $t$ and the choice of $c'_\lambda$.
\end{proof}

We can now use this identity by writing
\begin{equation}
\frac{\widehat P_\mu^{[n]}}{n!}\cdot\frac{c'_\mu}{P^*_\mu(\mu)}
=\sum_{\lambda\vdash n}\frac{c'_\mu}{c'_\lambda}
 \left[\begin{matrix}\lambda\\ \mu\end{matrix}\right]_{q,t}P_\lambda
=\sum_{\lambda\vdash n}t^{n(\mu)-n(\lambda)}Q_{\lambda/\mu}\left[\frac1{1-t}\right]P_\lambda
\end{equation}
and observing that, by the reproducing property of the Cauchy kernel,
\begin{equation}
Q_{\lambda/\mu}\left[\frac1{1-t}\right]
=\left\langle \sigma_1\left(\frac{X}{1-q}\right)P_\mu,Q_\lambda\right\rangle_{q,t},
\end{equation}
and
\[
 \frac{P_\mu^*(\mu)}{c'_\mu}
 =(-1)^{|\mu|}
   q^{n(\mu')}t^{-2n(\mu)},
\]
we arrive at
\begin{align}
\widehat{\mathbf P}_\mu&=\sum_{n\ge |\mu|}\frac{\widehat P_\mu^{[n]}}{n!}
=\frac{P^*_\mu(\mu)}{c'_\mu}\sum_\lambda t^{n(\mu)-n(\lambda)}
\left\langle \sigma_1\left(\frac{X}{1-q}\right)P_\mu,Q_\lambda\right\rangle_{q,t}P_\lambda\\
&=
(-1)^{|\mu|}t^{-2n(\mu)}q^{n(\mu')}
\sum_\lambda 
\left\langle \sigma_1\left(\frac{X}{1-q}\right)t^{n(\mu)} P_\mu,Q_\lambda\right\rangle_{q,t}
t^{-n(\lambda)}P_\lambda\\
&=\Dop\sigma_1\left(\frac{X}{1-q}\right) \boldsymbol\nabla P_\mu,\\
&=\Dop\,\Pcal_{1/(1-q)}\,\boldsymbol\nabla f,
\end{align}
where $\Dop$ is the linear operator defined by
\begin{equation}
\Dop P_\mu = t^{-n(\mu)}P_\mu,
\end{equation}
$\Pcal_Z$ denotes multiplication by $\sigma_1[ZX]$, 
and
\begin{equation}\label{eq:bold-nabla}
 \boldsymbol\nabla P_\mu
 =(-1)^{|\mu|}q^{n(\mu')}t^{-n(\mu)}P_\mu.
\end{equation}

We have therefore proved:
\begin{theorem}\label{thm:shifted-representative}
The exponentially normalized hat transform defined in
\eqref{eq:hat-exponential} satisfies
\begin{equation}\label{eq:shifted-representative}
\widehat{\mathbf f}=\Dop 
\Pcal_{1/(1-q)}
\boldsymbol\nabla f.
\end{equation}
Consequently the operator form of the hat transform is precisely
\eqref{eq:intro-Macdonald-transform}.
\end{theorem}

\subsection{The check transform}
The formula obtained above   shows that the hat transform is not
merely a spectral definition: after exponential normalization it is
a composition of three elementary linear operators.  Namely, put
\begin{equation}\label{eq:H-diagonal}
 \mathsf H_{q,t}P_\mu
 :=\frac{P_\mu^*(q^\mu;q,t)}{c'_\mu(q,t)}P_\mu.
\end{equation}
%%%%
Okounkov observes that the diagonal normalization of the interpolation
polynomials involves a choice.  It is the appropriate choice when one
wants the eigenvalues themselves to be his $P_\mu^*$, but it need not
be built into the transformation of the ordinary symmetric functions.
We accordingly define the \emph{check transform} by
\begin{equation}\label{eq:check-transform-definition}
  {
 \mathbf{\check f}
 :=\mathbf{\widehat{\mathsf H_{q,t}^{-1}f}}. }
\end{equation}
We denote its shifted eigenvalue by
\begin{equation}\label{eq:sharp-eigenvalue-definition}
  {
 f^\#:=\bigl(\mathsf H_{q,t}^{-1}f\bigr)^*. }
\end{equation}
Thus, for every $\lambda\vdash n$,
\begin{equation}\label{eq:check-spectral-property}
 \check f^{[n]}\times_{q,t}P_\lambda
 =f^\#(\lambda)P_\lambda.
\end{equation}
In particular, on the Macdonald basis these eigenvalues are the
normalized interpolation polynomials
\[
 \check P_\mu^*(\lambda)
 :=(-1)^{|\mu|}q^{-n(\mu')}t^{2n(\mu)}
   P_\mu^*(\lambda),
 \qquad
 \check P_\mu^*(\mu)=c'_\mu(q,t).
\]
The functions $p_\mu^\#$ form a filtered basis of the shifted
Macdonald algebra.  It is therefore natural to define coefficients
$d^{\#,\lambda}_{\mu\nu}(q,t)$ by
\begin{equation}\label{eq:check-structure-constants}
  {
 p_\mu^\#p_\nu^\#
 =\sum_\lambda d^{\#,\lambda}_{\mu\nu}(q,t)p_\lambda^\#,
 \qquad
 \check p_\mu\times_{q,t}\check p_\nu
 =\sum_\lambda d^{\#,\lambda}_{\mu\nu}(q,t)\check p_\lambda. }
\end{equation}
These are the direct $(q,t)$-analogues, in the unitriangular
normalization, of the stable coefficients $d_{\mu\nu}^{\lambda}$.
Theorem~\ref{thm:shifted-representative} gives
\begin{equation}\label{eq:hat-operator-announcement}
  {
 \mathbf{\widehat f}
 =\Dop\,M_{\sigma_1(X/(1-q))}\,\Dop^{-1}
 \mathsf H_{q,t}f. }
\end{equation}
At the same time it gives the simpler operator formula
\begin{equation}\label{eq:check-operator-announcement}
  {
 \mathbf{\check f}
 =\Dop\,M_{\sigma_1(X/(1-q))}\,\Dop^{-1}f. }
\end{equation}
If $f$ is homogeneous of degree $m$, taking the lowest homogeneous
component in this formula gives
\begin{equation}\label{eq:hat-lowest-component}
 [\mathbf{\widehat f}]_m=\mathsf H_{q,t}f,
 \qquad
 \widehat f^{[m]}=m!\,\mathsf H_{q,t}f.
\end{equation}
For the check transform, the corresponding identities are
\begin{equation}\label{eq:check-lowest-component}
  {
 [\mathbf{\check f}]_m=f,
 \qquad
 \check f^{[m]}=m!f. }
\end{equation}
The factorial in the second identity is solely due to the passage
from exponential normalization to the unnormalized homogeneous
components.

\subsection{Jack degeneration}
The check normalization also has the expected Jack degeneration,
provided that one rescales the variables in the ordinary, rather than
plethystic, sense.  To avoid ambiguity, for a scalar $\varepsilon$ we
write $\varepsilon X$ for the scalar dilation characterized by
\begin{equation}\label{eq:scalar-dilation}
 p_r(\varepsilon X)=\varepsilon^r p_r(X).
\end{equation}
In particular, the notation in the following proposition does not
mean the plethystic substitution
$p_r(X)\mapsto(1-q^r)p_r(X)$.

\begin{proposition}[Jack degeneration of the check transform]
\label{prop:check-Jack-degeneration}
Let $m=|\mu|$, set $q=t^\alpha$, and put $\varepsilon=1-q$.  Then,
in the degree completion,
\begin{equation}\label{eq:check-Jack-exponential-limit}
  {
 \lim_{t\to1}\varepsilon^{-m}
 \mathbf{\check P}_\mu(\varepsilon X;q,t)
 =P_\mu^{(\alpha)}(X)\exp(p_1). }
\end{equation}
Equivalently, for the nonexponentially normalized completed series,
\begin{equation}\label{eq:check-Jack-ordinary-limit}
  {
 \lim_{t\to1}\varepsilon^{-m}
 \check P_\mu(\varepsilon X;q,t)
 =\frac{m!P_\mu^{(\alpha)}(X)}{(1-p_1)^{m+1}}. }
\end{equation}
\end{proposition}

\begin{proof}
Let $S_\varepsilon$ denote the scalar dilation
$S_\varepsilon f(X)=f(\varepsilon X)$.  Since $\Dop$ preserves
homogeneous degree, it commutes with $S_\varepsilon$, while
$S_\varepsilon P_\mu=\varepsilon^mP_\mu$.  Formula
\eqref{eq:check-operator-announcement} therefore gives
\[
 \varepsilon^{-m}S_\varepsilon\mathbf{\check P}_\mu
 =\Dop\,
   M_{\sigma_1(\varepsilon X/(1-q))}\,
   \Dop^{-1}P_\mu(X;q,t).
\]
Now
\begin{align*}
 \sigma_1\left(\frac{\varepsilon X}{1-q}\right)
 &=\exp\left(
    \sum_{r\geq1}\frac{\varepsilon^r}{1-q^r}
                    \frac{p_r(X)}r\right),\\
 \frac{\varepsilon^r}{1-q^r}
 &=\frac{(1-q)^{r-1}}{1+q+\cdots+q^{r-1}}.
\end{align*}
The last expression tends to $1$ for $r=1$ and to $0$ for
$r>1$.  Hence the Cauchy factor tends coefficientwise to
$\exp(p_1)$.  Moreover,
\[
 \Dop P_\lambda=t^{-n(\lambda)}P_\lambda\longrightarrow P_\lambda,
 \qquad
 P_\mu(X;t^\alpha,t)\longrightarrow P_\mu^{(\alpha)}(X).
\]
This proves \eqref{eq:check-Jack-exponential-limit}.

The component of degree $m+r$ on its right-hand side is
\[
 P_\mu^{(\alpha)}\frac{p_1^r}{r!}.
\]
Passing back from exponential normalization multiplies this component
by $(m+r)!$.  Summing over $r\geq0$ gives
\[
 \sum_{r\geq0}\frac{(m+r)!}{r!}
 P_\mu^{(\alpha)}p_1^r
 =m!P_\mu^{(\alpha)}
   \sum_{r\geq0}\binom{m+r}{m}p_1^r
 =\frac{m!P_\mu^{(\alpha)}}{(1-p_1)^{m+1}},
\]
which proves \eqref{eq:check-Jack-ordinary-limit}.
\end{proof}

\section{The Nazarov--Sklyanin operators}
\subsection{Generating series}
Let $A^{(k)}$ be the Macdonald operators at infinity of
Nazarov--Sklyanin.  Their generating series is normalized by
\[
 A(u)=1+\sum_{k\geq1}\frac{A^{(k)}}{(u;t^{-1})_k},
 \qquad
 (u;t^{-1})_k=\prod_{r=0}^{k-1}(1-ut^{-r}),
\]
and
\begin{equation}\label{eq:NS-spectrum}
 A(u)P_\lambda
 =\prod_{i\geq1}
   \frac{q^{-\lambda_i}-ut^{1-i}}{1-ut^{1-i}}P_\lambda.
\end{equation}
Define $a_k(\lambda)$ by
\begin{equation}\label{eq:eigenvalue-definition}
 \prod_{i\geq1}
   \frac{q^{-\lambda_i}-ut^{1-i}}{1-ut^{1-i}}
 =1+\sum_{k\geq1}\frac{a_k(\lambda)}{(u;t^{-1})_k}.
\end{equation}
Thus $A^{(k)}P_\lambda=a_k(\lambda)P_\lambda$ and
$a_k(\lambda)=0$ if $\ell(\lambda)<k$.

We seek a completed series
\[
 \Fcal_k=\sum_{n\geq k}\Fcal_k^{[n]}
\]
such that multiplication by $\Fcal_k^{[n]}$ for
$\times_{q,t}$ coincides with $A^{(k)}$ in degree $n$.

\begin{proposition}[Spectral reconstruction]\label{prop:reconstruction}
There is a unique such series, and its homogeneous components are
\begin{equation}\label{eq:spectral-reconstruction}
 \Fcal_k^{[n]}
 =n!\sum_{\lambda\vdash n}
   \frac{a_k(\lambda)}{j_\lambda(q,t)}J_\lambda(X;q,t).
\end{equation}
Equivalently,
\begin{equation}\label{eq:spectral-reconstruction-P}
 \frac{\Fcal_k^{[n]}}{n!}
 =\sum_{\lambda\vdash n}
   a_k(\lambda)\frac{P_\lambda(X;q,t)}{c'_\lambda(q,t)}.
\end{equation}
\end{proposition}

\begin{proof}
If $F=\sum_{\lambda\vdash n}f_\lambda J_\lambda$, then
\eqref{eq:diagonal-product} shows that multiplication by $F$ has
eigenvalue $f_\lambda j_\lambda/n!$ on $J_\lambda$.  This proves
both existence and uniqueness.  The second expression follows from
$J_\lambda/j_\lambda=P_\lambda/c'_\lambda$.
\end{proof}

It is convenient to introduce the exponentially normalized series
\begin{equation}\label{eq:exponential-normalization}
 \mathbf F_k(X)
 :=\sum_{n\geq k}\frac{\Fcal_k^{[n]}(X)}{n!}.
\end{equation}

\subsection{The vacuum series}

Define
\begin{equation}\label{eq:vacuum}
 \Phi(X)
 :=\sum_\lambda\frac{P_\lambda(X;q,t)}{c'_\lambda(q,t)}.
\end{equation}
Recall that $\Dop$ is the diagonal operator
\begin{equation}\label{eq:Dt}
 \Dop P_\lambda(X;q,t)=t^{-n(\lambda)}P_\lambda(X;q,t),
 \qquad
 n(\lambda)=\sum_{i\geq1}(i-1)\lambda_i.
\end{equation}
We use the notation
\[
 \sigma_1(Z)
 =\operatorname{Exp}[Z]
 =\exp\left(\sum_{r\geq1}\frac{p_r(Z)}r\right)
 =\sum_{n\geq0}h_n(Z).
\]

\begin{proposition}\label{prop:vacuum}
The vacuum series has the closed form
\begin{equation}\label{eq:vacuum-closed}
 \Phi(X)=\Dop\,\sigma_1\left(\frac{X}{1-q}\right).
\end{equation}
Moreover,
\begin{equation}\label{eq:F-as-A-vacuum}
 \mathbf F_k=A^{(k)}\Phi.
\end{equation}
\end{proposition}

\begin{proof}
The Macdonald Cauchy kernel is
\[
 \Pi_{q,t}(X,Y)
 =\sigma_1\left(\frac{1-t}{1-q}XY\right)
 =\sum_\lambda P_\lambda(X;q,t)Q_\lambda(Y;q,t).
\]
The principal specialization is
\[
 P_\lambda\left[\frac1{1-t};q,t\right]
 =\frac{t^{n(\lambda)}}{c_\lambda(q,t)},
\]
and hence
\[
 Q_\lambda\left[\frac1{1-t};q,t\right]
 =\frac{t^{n(\lambda)}}{c'_\lambda(q,t)}.
\]
Specializing the second alphabet of the Cauchy kernel therefore gives
\[
 \sigma_1\left(\frac{X}{1-q}\right)
 =\sum_\lambda t^{n(\lambda)}
   \frac{P_\lambda(X;q,t)}{c'_\lambda(q,t)}.
\]
Application of $\Dop$ proves \eqref{eq:vacuum-closed}.  Finally,
$A^{(k)}$ and $\Dop$ are simultaneously diagonal in the Macdonald
basis; comparison with \eqref{eq:spectral-reconstruction-P} proves
\eqref{eq:F-as-A-vacuum}.
\end{proof}

\subsection{Reduction to a Hall--Littlewood sum}

Nazarov and Sklyanin proved the differential formula
\begin{equation}\label{eq:NS-differential}
 A^{(k)}
 =\sum_{\ell(\rho)=k}
   q^{-|\rho|}Q_\rho(X;t)P_\rho^\perp,
\end{equation}
where $P_\rho,Q_\rho$ are Hall--Littlewood functions of parameter
$t$, while $P_\rho^\perp$ denotes the adjoint of multiplication by
the Hall--Littlewood function $P_\rho(X;t)$ for the Macdonald scalar
product; thus
\[
 \langle P_\rho f,g\rangle_{q,t}
 =\langle f,P_\rho^\perp g\rangle_{q,t}.
\]
This notation is kept distinct both from Okounkov's interpolation
polynomial $P_\rho^*$ and from the  Foulkes derivative
$D_{P_\rho}$.  The reproducing property gives
\[
 P_\rho^\perp\sigma_1\left(\frac{X}{1-q}\right)
 =P_\rho\left[\frac1{1-t};t\right]
  \sigma_1\left(\frac{X}{1-q}\right).
\]
For Hall--Littlewood functions,
\[
 P_\rho\left[\frac1{1-t};t\right]
 =\frac{t^{n(\rho)}}{b_\rho(t)},
 \qquad Q_\rho(X;t)=b_\rho(t)P_\rho(X;t).
\]
The factors $b_\rho(t)$ cancel.

\begin{proposition}\label{prop:HL-reduction}
Put
\begin{equation}\label{eq:Theta-definition}
 \Theta_k(Z;t)
 :=\sum_{\ell(\rho)=k}t^{n(\rho)}P_\rho(Z;t).
\end{equation}
Then
\begin{equation}\label{eq:F-Theta}
  {
 \mathbf F_k(X)
 =\Dop\left[
   \sigma_1\left(\frac{X}{1-q}\right)
   \Theta_k\left(\frac Xq;t\right)
 \right]. }
\end{equation}
\end{proposition}

The factor $q^{-|\rho|}$ has thus been absorbed by the alphabet
change $X\mapsto X/q$.  The factor $t^{n(\rho)}$ is itself a
principal specialization:
\[
 Q_\rho\left[\frac1{1-t};t\right]=t^{n(\rho)}.
\]
Summing \eqref{eq:Theta-definition} over all lengths and using the
Hall--Littlewood Cauchy identity gives
\begin{equation}\label{eq:sum-all-Theta}
 \sum_{k\geq0}\Theta_k(Z;t)=\sigma_1(Z).
\end{equation}

\subsection{Evaluation of the fixed-length sum}

The specialization with an auxiliary variable $u$ is
\begin{equation}\label{eq:HL-u-specialization}
 Q_\rho\left[\frac{1-u}{1-t};t\right]
 =t^{n(\rho)}(u;t^{-1})_{\ell(\rho)}.
\end{equation}
The Hall--Littlewood Cauchy identity therefore yields the Newton
expansion
\begin{equation}\label{eq:Newton-Theta}
  {
 \sigma_1((1-u)X)
 =\sum_{k\geq0}(u;t^{-1})_k\Theta_k(X;t). }
\end{equation}

\begin{proposition}[Finite formula]\label{prop:finite-Theta}
For every $k\geq0$,
\begin{equation}\label{eq:finite-Theta}
  {
 \Theta_k(X;t)
 =\frac1{(t;t)_k}
  \sum_{j=0}^k(-1)^{k-j}t^{\binom{k-j}{2}}
  \begin{bmatrix}k\\j\end{bmatrix}_t
  \sigma_1((1-t^j)X), }
\end{equation}
where $(t;t)_k=\prod_{r=1}^k(1-t^r)$.
\end{proposition}

\begin{proof}
Evaluate \eqref{eq:Newton-Theta} successively at
$u=1,t,t^2,\ldots$.  Since $(t^m;t^{-1})_k=0$ for $k>m$, this gives
a triangular system.  Its inverse is the usual finite
$t$-binomial inversion, which is exactly \eqref{eq:finite-Theta}.
\end{proof}

For example,
\begin{align*}
 \Theta_1(X;t)
 &=\frac{\sigma_1((1-t)X)-1}{1-t},\\
 \Theta_2(X;t)
 &=\frac{\sigma_1((1-t^2)X)
 -(1+t)\sigma_1((1-t)X)+t}
 {(1-t)(1-t^2)},\\
 \Theta_3(X;t)
 &=\frac{\sigma_1((1-t^3)X)
 -(1+t+t^2)\sigma_1((1-t^2)X)}
 {(1-t)(1-t^2)(1-t^3)}\\
 &\quad+
 \frac{t(1+t+t^2)\sigma_1((1-t)X)-t^3}
 {(1-t)(1-t^2)(1-t^3)}.
\end{align*}

Combining Propositions~\ref{prop:HL-reduction} and
\ref{prop:finite-Theta} gives
\begin{equation}\label{eq:finite-F}
  {
 \begin{aligned}
 \mathbf F_k(X)
 =\frac1{(t;t)_k}\Dop
 \sum_{j=0}^k&(-1)^{k-j}t^{\binom{k-j}{2}}
 \begin{bmatrix}k\\j\end{bmatrix}_t\\[-1mm]
 &\times\sigma_1\left(
 \frac{1-(1-q)t^j}{q(1-q)}X
 \right).
 \end{aligned} }
\end{equation}

\subsection{A generating series for all columns}

There is a more economical form of \eqref{eq:finite-Theta}.  Write
\[
 \sigma_1\left(-\frac{t^jX}{q}\right)
 =\sum_{n\geq0}h_n\left(-\frac Xq\right)t^{jn}.
\]
The finite $t$-binomial theorem gives
\begin{align*}
 &\sum_{j=0}^k(-1)^{k-j}t^{\binom{k-j}{2}}
 \begin{bmatrix}k\\j\end{bmatrix}_t t^{jn}
 =\prod_{r=0}^{k-1}(t^n-t^r).
\end{align*}
This vanishes for $n<k$, whereas for $n\geq k$ it equals
\[
 (-1)^kt^{\binom{k}{2}}
 \frac{(t;t)_n}{(t;t)_{n-k}}.
\]
Hence
\begin{equation}\label{eq:Theta-binomial}
  {
 \Theta_k\left(\frac Xq;t\right)
 =(-1)^kt^{\binom{k}{2}}
  \sigma_1\left(\frac Xq\right)
  \sum_{n\geq k}
  \begin{bmatrix}n\\k\end{bmatrix}_t
  h_n\left(-\frac Xq\right). }
\end{equation}
Since
\[
 h_n\left(-\frac Xq\right)=(-1)^nq^{-n}e_n(X),
\]
this may equivalently be written in the elementary basis.

\begin{proposition}[Master generating series]\label{prop:master}
One has
\begin{equation}\label{eq:master-Theta}
  {
 \sum_{k\geq0}z^k\Theta_k\left(\frac Xq;t\right)
 =\sigma_1\left(\frac Xq\right)
  \sum_{n\geq0}(z;t)_n
  h_n\left(-\frac Xq\right). }
\end{equation}
Consequently,
\begin{equation}\label{eq:master-F}
  {
 \sum_{k\geq0}z^k\mathbf F_k(X)
 =\Dop\left[
  \sigma_1\left(\frac{X}{q(1-q)}\right)
  \sum_{n\geq0}(z;t)_n
  h_n\left(-\frac Xq\right)
 \right]. }
\end{equation}
Equivalently,
\begin{equation}\label{eq:master-F-elementary}
  {
 \sum_{k\geq0}z^k\mathbf F_k(X)
 =\Dop\left[
  \sigma_1\left(\frac{X}{q(1-q)}\right)
  \sum_{n\geq0}(-1)^nq^{-n}(z;t)_n e_n(X)
 \right]. }
\end{equation}
\end{proposition}

\begin{proof}
Multiply \eqref{eq:Theta-binomial} by $z^k$ and sum over $k$.
The finite $t$-binomial theorem gives
\[
 \sum_{k=0}^n(-1)^kt^{\binom{k}{2}}
 \begin{bmatrix}n\\k\end{bmatrix}_t z^k=(z;t)_n.
\]
This proves \eqref{eq:master-Theta}.  Formula \eqref{eq:master-F}
then follows from \eqref{eq:F-Theta} and
\[
 \sigma_1\left(\frac{X}{1-q}\right)
 \sigma_1\left(\frac Xq\right)
 =\sigma_1\left(\frac{X}{q(1-q)}\right).
\]
\end{proof}

\subsection{Nazarov--Sklyanin eigenvalues}
We can now identify exactly the shifted functions occurring in the
Nazarov--Sklyanin spectrum.  Write
\[
 e_k^*(x;T):=P_{(1^k)}^*(x;Q,T).
\]
As is clear either from the tableau formula or from the following
generating series, this function is independent of the first
Macdonald parameter $Q$.

\begin{proposition}\label{prop:NS-shifted-column}
The eigenvalue of $A^{(k)}$ on $P_\lambda(X;q,t)$ is
\begin{equation}\label{eq:NS-shifted-column}
  {
 a_k(\lambda)
 =t^{-\binom{k}{2}}
  P_{(1^k)}^*\left(q^{-\lambda};q^{-1},t^{-1}\right). }
\end{equation}
Consequently,
\begin{equation}\label{eq:F-shifted-column}
  {
 \frac{\Fcal_k^{[n]}}{n!}
 =t^{-\binom{k}{2}}
  \sum_{\lambda\vdash n}
  \frac{P_{(1^k)}^*(q^{-\lambda};q^{-1},t^{-1})}
       {c'_\lambda(q,t)}P_\lambda(X;q,t). }
\end{equation}
\end{proposition}

\begin{proof}
Okounkov's generating series for the shifted elementary functions is
\begin{equation}\label{eq:Okounkov-column-generating}
 \prod_{i\geq1}
 \frac{1+x_iT^{1-i}/v}{1+T^{1-i}/v}
 =\sum_{k\geq0}
 \frac{e_k^*(x;T)}
 {(v+1)(v+T^{-1})\cdots(v+T^{1-k})}.
\end{equation}
Set $T=t^{-1}$, $x_i=q^{-\lambda_i}$ and $v=-u$.  The
left-hand side becomes
\[
 \prod_{i\geq1}
 \frac{1-q^{-\lambda_i}t^{i-1}/u}
      {1-t^{i-1}/u}
 =\prod_{i\geq1}
 \frac{q^{-\lambda_i}-ut^{1-i}}
      {1-ut^{1-i}},
\]
which is the eigenvalue in \eqref{eq:NS-spectrum}.  On the other
hand,
\[
 \prod_{r=0}^{k-1}(v+T^{-r})
 =\prod_{r=0}^{k-1}(t^r-u)
 =t^{\binom{k}{2}}(u;t^{-1})_k.
\]
Comparison with \eqref{eq:eigenvalue-definition} proves
\eqref{eq:NS-shifted-column}; the spectral reconstruction formula
then gives \eqref{eq:F-shifted-column}.
\end{proof}

Thus the operators at infinity are precisely the diagonal operators
associated with the shifted columns at inverted parameters, up to the
universal factor $t^{-\binom{k}{2}}$.  The master series
\eqref{eq:master-F} is the corresponding closed expression in the
ordinary Macdonald basis with parameters $(q,t)$; the inversion of
the parameters cannot simply be suppressed because the spectral
weight remains $1/c'_\lambda(q,t)$.

\subsection{Initial terms and interpretation}

The unique partition of size and length $k$ is $(1^k)$.  Thus
\[
 [\Theta_k(X/q;t)]_k
 =q^{-k}t^{\binom{k}{2}}e_k(X).
\]
Since $\Dop e_k=t^{-\binom{k}{2}}e_k$, we recover
\begin{equation}\label{eq:leading-term}
  {\Fcal_k^{[k]}=k!q^{-k}e_k.}
\end{equation}
The next component is
\begin{equation}\label{eq:first-correction}
 \Fcal_k^{[k+1]}
 =\frac{(k+1)!}{q^{2k}(1-q)}
   \mathsf D_q(p_1e_k),
\end{equation}
where, on the partitions occurring in this Pieri product,
\[
 \mathsf D_qP_\lambda=q^{|\lambda|-\lambda_1}P_\lambda.
\]
Formula \eqref{eq:master-F} explains why the higher corrections are
not obtained by multiplying $e_k$ by an ordinary symmetric series:
already the following components contain terms $p_3,p_4,\ldots$.
They arise automatically from the second factor in
\eqref{eq:master-F}, rather than from an ad hoc multiplicative
correction.

Formula \eqref{eq:NS-shifted-column} now gives an intrinsic shifted
interpretation of $\Fcal_k$: it represents the interpolation column
$P_{(1^k)}^*$ at inverted parameters, multiplied by
$t^{-\binom{k}{2}}$.  We still refrain from denoting this series by
$\widehat J_{1^k}$, since that notation would also require a chosen
normalization for arbitrary partitions.  Independently of notation,
$\Fcal_k$ is the unique completed series whose multiplication for the
product dual to \eqref{eq:diagonal-coproduct} realizes $A^{(k)}$.

\section{Examples}

Low-degree calculations are immediate from the definition.  In degree
one,
\begin{equation}\label{eq:p1-product}
 p_1\times_{q,t}p_1=(1-q)p_1.
\end{equation}
For degree two it is most economical to give the answers in the
orthogonal $J$-basis.  With $p_{11}=p_1^2$, one obtains
\begin{align}
 p_2\times_{q,t}p_2
 &=A_2J_{(2)}+B_{22}J_{(1,1)},
 \label{eq:p2-p2-product}\\
 p_{11}\times_{q,t}p_2
 &=A_2\bigl(J_{(2)}-J_{(1,1)}\bigr),
 \label{eq:p11-p2-product}\\
 p_{11}\times_{q,t}p_{11}
 &=A_2J_{(2)}+B_{11}J_{(1,1)},
 \label{eq:p11-p11-product}
\end{align}
where
\begin{align*}
 A_2&=\frac{(1-q)(1-q^2)}{2(1-t)(1-qt)},\\
 B_{22}&=\frac{(1-q^2)(1+q)}{2(1+t)(1-qt)},\\
 B_{11}&=\frac{(1-q)^3(1+t)}{2(1-t)^2(1-qt)}.
\end{align*}
For example, these identities follow by inserting
\begin{align*}
 J_{(2)}
 &=\frac{(1-t)^2(1+q)}2p_{11}
   +\frac{(1-q)(1-t^2)}2p_2,\\
 J_{(1,1)}
 &=\frac{(1-t)(1-t^2)}2(p_{11}-p_2)
\end{align*}
into \eqref{eq:diagonal-product}.

 For completeness, put
$K=(1-q)(1-q^2)$.  The four scalar products and the two norms needed
for the calculation are
\[
\begin{array}{c|cc}
 &J_{(2)}&J_{(1,1)}\\ \hline
 \langle p_2,\mathord\cdot\rangle_{q,t}
   &K&-(1-q^2)(1-t)\\
 \langle p_{11},\mathord\cdot\rangle_{q,t}
   &K&(1-q)^2(1+t)
\end{array}
\]
and
\begin{align*}
 j_{(2)}&=(1-t)(1-qt)(1-q)(1-q^2),\\
 j_{(1,1)}&=(1-t)(1-t^2)(1-q)(1-qt).
\end{align*}
Substitution in
\[
 F\times_{q,t}G
 =\sum_{\lambda\vdash2}
   \frac{\langle F,J_\lambda\rangle_{q,t}
         \langle G,J_\lambda\rangle_{q,t}}
        {2j_\lambda(q,t)}J_\lambda
\]
then verifies every coefficient in \eqref{eq:p2-p2-product}--
\eqref{eq:p11-p11-product} directly.

The stable example already displays the new phenomenon. 

By the spectral characterization of the Macdonald hat transform,
in every degree $n\geq2$,
\begin{equation}\label{eq:hat-p2-square-spectral}
  {
 \widehat p_2^{[n]}\times_{q,t}\widehat p_2^{[n]}
 =n!\sum_{\lambda\vdash n}
   \frac{p_2^*(\lambda)^2}{j_\lambda(q,t)}J_\lambda. }
\end{equation}
Thus this square is the stable representative of the pointwise square
$(p_2^*)^2$.  For $n=2$, formula (32) of
\cite{BenDaliDAdderio}, or a direct expansion of the two integral
Macdonald polynomials, gives
\begin{align*}
 p_2^*((2))
 &=q(1-q)(1-q^2),\\
 p_2^*((1,1))
 &=-t^{-2}(1-t)(1-q^2),
\end{align*}
The factors $q$ and $t^{-2}$ in these two evaluations are important.
Indeed, comparison with
%%%
\[
 p_2=
 \frac{1}{(1-t)(1-qt)}J_{(2)}
 -\frac{1+q}{(1-t^2)(1-qt)}J_{(1,1)}
\]
gives
\[
  {
 \widehat p_2^{[2]}
 =\frac{2q}{(1-t)(1-qt)}J_{(2)}
 -\frac{2t^{-2}(1+q)}{(1-t^2)(1-qt)}J_{(1,1)}. }
\]
%%%
Equivalently, for the exponentially normalized series introduced
below,
\begin{equation}\label{eq:p2-minimal-component-H}
  {
 [\mathbf{\widehat p}_2]_2
 =\frac{\widehat p_2^{[2]}}{2!}
 =\mathsf H_{q,t}(p_2)
 =\frac{q}{(1-t)(1-qt)}J_{(2)}
  -\frac{t^{-2}(1+q)}{(1-t^2)(1-qt)}J_{(1,1)}. }
\end{equation}
%%%
Thus $\widehat p_2^{[2]}$ is not a scalar multiple of $p_2$.
This is not a defect of the product normalization: it is the diagonal
twist inherent in Okounkov's normalization of interpolation Macdonald
polynomials.  More generally, as will follow from
\eqref{eq:shifted-diagonal-evaluation-comparison},
\begin{equation}\label{eq:hat-leading-component}
 \widehat P_\mu^{[|\mu|]}
 =|\mu|!(-1)^{|\mu|}q^{n(\mu')}t^{-2n(\mu)}P_\mu.
\end{equation}
This convention is retained throughout because it agrees directly with
the Macdonald characters of Ben Dali and D'Adderio.  Squaring the two
spectral coefficients now gives
\begin{align}
 \widehat p_2^{[2]}\times_{q,t}\widehat p_2^{[2]}
 ={}&\frac{2q^2(1-q)(1-q^2)}{(1-t)(1-qt)}J_{(2)}
 \notag\\
 &+\frac{2t^{-4}(1-q^2)(1+q)}{(1+t)(1-qt)}J_{(1,1)}.
 \label{eq:hat-p2-square-degree-two}
\end{align}
Formula \eqref{eq:hat-p2-square-degree-two} is only the degree-two
component of a finite stable identity.  Put
\begin{align*}
 C_3&=\frac{(q-1)(q+1)^2(t^2+t+1)}{3t^2},\\
 C_{21}&=\frac{(q-1)(q+1)^2(t-1)(t+1)}{2t^2},\\
 C_{111}&=\frac{(q-1)(q+1)^2(t-1)^2}{6t^2},\\
 C_{11}&=\frac{(q-1)(q+1)^2(t-1)(qt-t+1)}{2t^2},\\
 C_2&=\frac{(q-1)(q+1)}{2t^2}
 \bigl(q^2t^2+q^2t-2qt^2-qt+q+t^2-2t+1\bigr).
\end{align*}
Then
\begin{equation}\label{eq:hat-p2-square-stable-Macdonald}
  {
 \widehat p_2\times_{q,t}\widehat p_2
 =\widehat p_{22}+C_3\widehat p_3
  +C_{21}\widehat p_{21}+C_{111}\widehat p_{111}
  +C_2\widehat p_2+C_{11}\widehat p_{11}. }
\end{equation}
Here and below, an identity between completed series means the
componentwise identity in every homogeneous degree.

To verify \eqref{eq:hat-p2-square-stable-Macdonald}, the spectral
characterization reduces it to
\begin{equation}\label{eq:p2-star-square}
 (p_2^*)^2=p_{22}^*+C_3p_3^*+C_{21}p_{21}^*
 +C_{111}p_{111}^*+C_2p_2^*+C_{11}p_{11}^*.
\end{equation}
Indeed, the left-hand side has shifted degree four and top homogeneous
component $p_{22}$.  After subtracting $p_{22}^*$, only characters of
sizes two and three can occur.  Evaluation on $(2)$ and $(1,1)$ first
determines $C_2,C_{11}$.  Evaluation on $(3),(2,1),(1,1,1)$ then
determines $C_3,C_{21},C_{111}$.  The required values in equal size
are obtained directly from
\[
 (-1)^{|\lambda|}q^{n(\lambda')}t^{-2n(\lambda)}J_\lambda
 =\sum_{\mu\vdash|\lambda|}
   \frac{p_\mu^*(\lambda)}{z_\mu(q,t)}p_\mu.
\]
Substitution gives the five coefficients displayed above.  The
difference between the two sides of \eqref{eq:p2-star-square} has
shifted degree at most four, vanishes on every diagram of size less
than four, and has zero top homogeneous component.  The
characterization theorem for Macdonald characters therefore makes it
identically zero.

Thus finite stability survives for generic $(q,t)$.  What changes from
the classical and Jack cases is not finiteness, but the explicit form
of the representatives: their homogeneous components contain the
nonmultiplicative diagonal twist visible already in
\eqref{eq:hat-leading-component}.

Let us check explicitly that \eqref{eq:hat-p2-square-stable-Macdonald}
degenerates to \eqref{eq:intro-Jack-square}.  Set
\begin{equation}\label{eq:Jack-degeneration-parameters}
 t=1+\gamma,\qquad q=t^\alpha=(1+\gamma)^\alpha,
 \qquad \gamma\longrightarrow0.
\end{equation}
The normalization compatible with the stable Jack characters used in
the introduction is
\begin{equation}\label{eq:Macdonald-to-Jack-character-normalization}
 \operatorname{Ch}^{(\alpha)}_\mu
 =\lim_{\gamma\to0}
  \alpha^{-\ell(\mu)}\gamma^{-|\mu|}p_\mu^*.
\end{equation}
The factor $\alpha^{-\ell(\mu)}$ is essential.  Indeed,
$z_\mu(q,t)\to z_\mu\alpha^{\ell(\mu)}$, so that
\eqref{eq:Macdonald-to-Jack-character-normalization} is $z_\mu$ times
the Jack character in the normalization of Ben Dali and D'Adderio.

After dividing \eqref{eq:p2-star-square} by
$\alpha^2\gamma^4$, the effective coefficient of the character
indexed by $\mu$ is
\[
 \alpha^{\ell(\mu)-2}\gamma^{|\mu|-4}C_\mu,
\]
where $C_{22}=1$.  Direct expansion gives
\begin{align*}
 \lim_{\gamma\to0}\alpha^{-1}\gamma^{-1}C_3&=4,
 &\lim_{\gamma\to0}\gamma^{-1}C_{21}&=0,
 &\lim_{\gamma\to0}\alpha\gamma^{-1}C_{111}&=0,\\
 \lim_{\gamma\to0}\alpha^{-1}\gamma^{-2}C_2&=2(\alpha-1),
 &\lim_{\gamma\to0}\gamma^{-2}C_{11}&=2\alpha.
\end{align*}
Consequently \eqref{eq:hat-p2-square-stable-Macdonald} tends exactly
to
\[
 \widehat p_2\times_\alpha\widehat p_2
 =\widehat p_{22}+4\widehat p_3
  +2\alpha\widehat p_{11}+2(\alpha-1)\widehat p_2,
\]
as asserted in \eqref{eq:intro-Jack-square}.  In particular, the two
additional Macdonald terms indexed by $(2,1)$ and $(1,1,1)$ disappear
in the Jack degeneration.

One may alternatively use the linear path $t=1+\gamma$,
$q=1+\alpha\gamma$.  It gives the same five limits: the calculation
above uses the full expansion of $(1+\gamma)^\alpha$ and therefore
checks the standard degeneration $q=t^\alpha$, not merely its
linearization.
For comparison, the renormalized transform announced in the
introduction has the particularly simple lowest component
\begin{equation}\label{eq:check-p2-minimal-component}
  {
 \check p_2^{[2]}=2p_2
 =\frac{2}{(1-t)(1-qt)}J_{(2)}
  -\frac{2(1+q)}{(1-t^2)(1-qt)}J_{(1,1)}. }
\end{equation}
Unlike \eqref{eq:p2-minimal-component-H}, this formula contains no
diagonal monomial twist.
Its product in the smallest homogeneous degree is consequently
\begin{equation}\label{eq:check-p2-square-degree-two}
  {
 \begin{aligned}
 \check p_2^{[2]}\times_{q,t}\check p_2^{[2]}
 ={}&\frac{2(1-q)(1-q^2)}{(1-t)(1-qt)}J_{(2)}\\
 &+\frac{2(1-q^2)(1+q)}{(1+t)(1-qt)}J_{(1,1)}.
 \end{aligned} }
\end{equation}
This is the degree-two component of the stable product
$\check p_2\times_{q,t}\check p_2$.  The complete expansion in the
check power-sum basis is finite, but, unlike
\eqref{eq:hat-p2-square-stable-Macdonald},
its top filtered component already involves all five partitions of
four.  Indeed it is obtained by expanding
\[
 \mathsf H_{q,t}\bigl((\mathsf H_{q,t}^{-1}p_2)^2\bigr)
\]
in the power-sum basis.  Since $\mathsf H_{q,t}$ is not multiplicative,
the resulting rational coefficients are substantially less compact
than those in \eqref{eq:hat-p2-square-stable-Macdonald}; we therefore retain
\eqref{eq:check-p2-square-degree-two} as the readable first example.

\subsection{The first nontrivial character in content form}

The eigenvalue corresponding to $\widehat p_2$ admits a particularly
simple expression.  For a cell $s=(i,j)$, put
\begin{equation}\label{eq:qt-content}
 \chi_{q,t}(s):=q^{a'(s)}t^{-l'(s)}=q^{j-1}t^{1-i},
\end{equation}
and define the content power sums
\begin{equation}\label{eq:content-power-sums}
 B_r(\lambda):=\sum_{s\in\lambda}\chi_{q,t}(s)^r.
\end{equation}

\begin{proposition}\label{prop:p2-content-eigenvalue}
For every partition $\lambda$,
\begin{equation}\label{eq:p2-content-eigenvalue}
  {
 p_2^*(\lambda)
 =(1-q^2)\sum_{s\in\lambda}
  \chi_{q,t}(s)\bigl(1-\chi_{q,t}(s)\bigr)
 =(1-q^2)\bigl(B_1(\lambda)-B_2(\lambda)\bigr). }
\end{equation}
Consequently the multiplication operator associated with
$\widehat p_2$ satisfies
\begin{equation}\label{eq:hat-p2-content-action}
 \widehat p_2^{[n]}\times_{q,t}Q_\lambda
 =(1-q^2)\bigl(B_1(\lambda)-B_2(\lambda)\bigr)Q_\lambda,
 \qquad \lambda\vdash n.
\end{equation}
\end{proposition}

\begin{proof}
Write $v_i=q^{\lambda_i}$.  Summing first along the rows of $\lambda$
shows that the right-hand side of \eqref{eq:p2-content-eigenvalue} is
\begin{equation}\label{eq:p2-shifted-variables}
 \sum_{i\geq1}\left\{
 (1+q)t^{1-i}(1-v_i)
 -t^{2(1-i)}(1-v_i^2)\right\}.
\end{equation}
This expression is symmetric in the shifted variables
$v_i t^{1-i}$, is compatible with adjoining a last variable equal to
$1$, and has degree two.  Its top homogeneous component is
\[
 \sum_{i\geq1}t^{2(1-i)}v_i^2
 =p_2(v_1,t^{-1}v_2,t^{-2}v_3,\ldots).
\]
It vanishes on the two diagrams of size strictly less than two,
$\varnothing$ and $(1)$.  The characterization theorem for Macdonald
characters in \cite{BenDaliDAdderio} therefore identifies
\eqref{eq:p2-shifted-variables} with the Macdonald character
$\widetilde\theta_{(2)}^{(q,t)}=p_2^*$.
Equation \eqref{eq:hat-p2-content-action} then follows from
Proposition~\ref{prop:hat-characterization}.
\end{proof}

There is also a concise operator formulation.  Let
$\boldsymbol\Delta_f$ denote the integral Macdonald Delta operator
defined by
\begin{equation}\label{eq:integral-Delta-f}
 \boldsymbol\Delta_fQ_\lambda
 =f\bigl[\{\chi_{q,t}(s):s\in\lambda\}\bigr]Q_\lambda.
\end{equation}
Then the multiplication operator by $\widehat p_2$ is
\begin{equation}\label{eq:hat-p2-Delta-operator}
  {
 \Delta_{\widehat p_2}^{(q,t)}
 =(1-q^2)
  \left(\boldsymbol\Delta_{p_1}-\boldsymbol\Delta_{p_2}\right)
 =(1-q^2)\left(
  \boldsymbol\Delta_{e_1}
  -(\boldsymbol\Delta_{e_1})^{,2}
  +2\boldsymbol\Delta_{e_2}\right). }
\end{equation}
Equivalently, if
\begin{equation}\label{eq:integral-Delta-generating}
 \boldsymbol\Delta_vQ_\lambda
 =\prod_{s\in\lambda}(1-v\chi_{q,t}(s))Q_\lambda,
\end{equation}
then $B_r=-r[v^r]\log\boldsymbol\Delta_v$ and
\eqref{eq:hat-p2-Delta-operator} is obtained from the first two
coefficients of $\log\boldsymbol\Delta_v$.

Formula \eqref{eq:hat-p2-Delta-operator} is more natural than an
expansion in the differential blocks
\[
 (I\mid J)=\prod_{i\in I}p_i\prod_{j\in J}D_{p_j}.
\]
For generic $(q,t)$, the Macdonald Delta operators are difference, or
vertex, operators and their normally ordered block expansions contain
terms of unbounded differential order.  The finite cut-and-join
operator reappears only after taking the Jack degeneration with the
appropriate normalization.  Thus a block formula exists as a formal
infinite expansion, but \eqref{eq:p2-content-eigenvalue} and
\eqref{eq:hat-p2-Delta-operator} retain considerably more of the
structure.

\section{Comparison with Macdonald characters and Theta operators}

We now compare the preceding construction with that of Ben Dali and
D'Adderio~\cite{BenDaliDAdderio}.  The comparison 
requires some care with notation.  In the literature on Theta
operators one sets
\[
 M=(1-q)(1-t),\qquad f^{\mathrm{pl}}[X]:=f[X/M].
\]
The superscript usually denoted by a star in
\cite{DAdderioIraciVandenWyngaerd,DAdderioRomero} means this
\emph{plethystic substitution}; it is unrelated to Okounkov's shifted
polynomial $f^*$.  We use the superscript ${\rm pl}$ in this section
to prevent this collision.

Let $\widetilde H_\lambda$ be the modified Macdonald basis, and set
\begin{equation}\label{eq:Pi-operator}
 \mathbf\Pi\widetilde H_\lambda
 =\Pi_\lambda\widetilde H_\lambda,
 \qquad
 \Pi_\lambda=
 \prod_{s\in\lambda\setminus\{(1,1)\}}
 (1-q^{a'(s)}t^{l'(s)}).
\end{equation}
Apart from the customary exceptional convention on constants, the
Theta operator attached to a homogeneous symmetric function $f$ is
\begin{equation}\label{eq:standard-Theta-definition}
 \Theta_f=\mathbf\Pi\,M_{f^{\mathrm{pl}}}\,\mathbf\Pi^{-1}.
\end{equation}
This is the definition introduced in
\cite{DAdderioIraciVandenWyngaerd} and used in
\cite{DAdderioRomero}.  If
\[
 \Delta_v=\sum_{r\geq0}(-v)^r\Delta_{e_r},
\]
then the latter paper packages the elementary Theta operators into
\begin{equation}\label{eq:Theta-generating}
 \widetilde\Theta(z,v)
 =\Delta_v\Pcal_{-z/M}\Delta_v^{-1},
 \qquad
 \Pcal_ZF[X]=\sigma_1[ZX]F[X].
\end{equation}
Indeed, the coefficient of $z^r$, followed by the regular limit
$v\to1$, is $(-1)^r\Theta_{e_r}$.

Ben Dali and D'Adderio introduce, first in the modified realization,
\begin{equation}\label{eq:BDD-Gamma-modified}
 \Gamma(u,v)=
 \Delta_{1/v}\Pcal_{uv/(1-q)}\Delta_{1/v}^{-1}.
\end{equation}
It is not merely analogous to \eqref{eq:Theta-generating}: they prove
the factorization
\begin{equation}\label{eq:Gamma-Theta-factorization}
  {
 \Gamma(u,v)=
 \widetilde\Theta(uv,1/v)^{-1}
 \widetilde\Theta(tuv,1/v). }
\end{equation}
Their integral-form operator $\boldsymbol\Gamma$ is obtained by the
standard change of alphabet between modified and integral Macdonald
polynomials; in that realization its middle multiplication factor is
$\Pcal_{uv(1-t)/(1-q)}$.

The main point for us is their explicit isomorphism from ordinary to
shifted symmetric functions.  If $\boldsymbol\nabla$ denotes the
integral-form nabla operator
\begin{equation}\label{eq:bold-nabla}
 \boldsymbol\nabla P_\mu
 =(-1)^{|\mu|}q^{n(\mu')}t^{-n(\mu)}P_\mu,
\end{equation}
then their formula reads
\begin{equation}\label{eq:BDD-shifted-pairing}
 f^*(\lambda)=
 \left\langle
  \Pcal_{1/(1-q)}\boldsymbol\nabla f,
  t^{-n(\lambda)}J_\lambda
 \right\rangle_{q,t}.
\end{equation}
This immediately identifies their construction with our hat
transform.

\begin{proposition}\label{prop:BDD-hat-comparison}
For every $f\in\Sym$, the exponentially normalized stable
representative of its shifted function is
\begin{equation}\label{eq:BDD-hat-comparison}
  {
 \mathbf{\widehat f}
 =\Dop\,\Pcal_{1/(1-q)}\boldsymbol\nabla f. }
\end{equation}
Consequently, the hat transform is the spectral realization, for the
product $\times_{q,t}$, of the Ben Dali--D'Adderio isomorphism
$f\mapsto f^*$.
\end{proposition}

\begin{proof}
Write
\[
 \Pcal_{1/(1-q)}\boldsymbol\nabla f
 =\sum_\lambda a_\lambda P_\lambda.
\]
Since $\langle P_\lambda,J_\lambda\rangle_{q,t}=c'_\lambda$,
formula \eqref{eq:BDD-shifted-pairing} gives
\[
 f^*(\lambda)=t^{-n(\lambda)}a_\lambda c'_\lambda,
 \qquad
 a_\lambda=
 \frac{t^{n(\lambda)}f^*(\lambda)}{c'_\lambda}.
\]
Applying $\Dop P_\lambda=t^{-n(\lambda)}P_\lambda$ yields
\[
 \Dop\Pcal_{1/(1-q)}\boldsymbol\nabla f
 =\sum_\lambda\frac{f^*(\lambda)}{c'_\lambda}P_\lambda,
\]
which is exactly \eqref{eq:hat-exponential-spectral} and its linear
extension.
\end{proof}

This also checks, without any further calculation, the operator
factorization obtained in Theorem~\ref{thm:shifted-representative}.
Indeed Okounkov's diagonal evaluation is
\begin{equation}\label{eq:shifted-diagonal-evaluation-comparison}
 \frac{P_\mu^*(q^\mu;q,t)}{c'_\mu(q,t)}
 =(-1)^{|\mu|}q^{n(\mu')}t^{-2n(\mu)}.
\end{equation}
Therefore
\begin{equation}\label{eq:H-nabla-comparison}
 \Dop^{-1}\mathsf H_{q,t}=\boldsymbol\nabla,
\end{equation}
and \eqref{eq:hat-operator-announcement} becomes precisely
\eqref{eq:BDD-hat-comparison}.

Finally, Ben Dali and D'Adderio define their Macdonald characters by
iterating $\boldsymbol\Gamma$ and pairing with $p_\mu$.  Their
Theorem~4.3 says equivalently that
\begin{equation}\label{eq:BDD-character-shifted-power}
 \widetilde\theta_\mu^{(q,t)}=p_\mu^*.
\end{equation}
In our language this becomes the concrete spectral statement
\begin{equation}\label{eq:hat-character-comparison}
 \widehat p_\mu^{[n]}\times_{q,t}Q_\lambda
 =\widetilde\theta_\mu^{(q,t)}(\lambda)Q_\lambda
 \qquad(\lambda\vdash n).
\end{equation}
Thus the two constructions produce the same shifted functions.  The
roles are complementary: the Theta formalism constructs and studies
the shifted functions through creation operators, whereas the diagonal
coproduct turns their values into multiplication eigenvalues and
selects a unique stable representative.  Our derivation above, from
Okounkov's binomial formula, and our analysis of the
Nazarov--Sklyanin operators, are 
independent of the Theta-operator argument; the comparison is an
a posteriori identification.

\section{Conclusion and perspectives}

The  diagonal Macdonald coproduct therefore leads to a natural
two-parameter continuation of the stable-series formalism.  Its
completed multiplication algebra realizes Okounkov's shifted
Macdonald functions as eigenvalues, the Nazarov--Sklyanin operators
correspond to shifted columns at inverted parameters, and the whole
construction agrees with the Macdonald characters of Ben Dali and
D'Adderio.  The essential difference from the classical and Jack
theories is that the stable transform is no longer multiplication by
a universal rational series: it involves Cauchy multiplication
together with Macdonald diagonal operators.

Several questions are left open.  First, the Delta-operator expression
\eqref{eq:hat-p2-Delta-operator} suggests that normally ordered
expansions in the differential blocks $(I\mid J)$ should be obtained
most naturally through the shuffle-algebra and vertex-operator
formalisms of \cite{NegutOperators,BenDaliBonzomDolega}.  Such
expansions are expected to have unbounded differential order; their
Jack limits should recover the finite cut-and-join operators.  Second,
the structure constants of Macdonald characters introduced in
\cite{BenDaliDAdderio} are, in the present realization, precisely the
stable structure constants of the product $\times_{q,t}$.  This may
provide a useful operator interpretation of their positivity
conjectures and of their relation with the super-nabla operator.
Finally, Cuenca's interpolation Macdonald operators at infinity
\cite{CuencaInterpolation} involve Hall--Littlewood Cauchy identities
closely related to those used here.  It would be interesting to find
an explicit intertwining relation, presumably governed by Okounkov's
binomial transform, between his operators on interpolation Macdonald
functions and the multiplication operators constructed in this
paper.

\section*{Acknowledgments}
The author used ChatGPT (OpenAI) as an interactive aid in organizing
the manuscript, editing the English text, and checking
computer-algebra calculations.  All mathematical statements and
conclusions remain the responsibility of the author.

\end{document}